\documentclass[english]{amsart}
\usepackage[T1]{fontenc}
\usepackage[latin9]{inputenc}
\usepackage{verbatim}
\usepackage{amsthm}
\usepackage{amssymb}

\makeatletter
\theoremstyle{plain}
\newtheorem{thm}{Theorem}[section]
  \theoremstyle{definition}
  \newtheorem{defn}[thm]{Definition}
  \theoremstyle{plain}
  \newtheorem{prop}[thm]{Proposition}
  \theoremstyle{plain}
  \newtheorem{lem}[thm]{Lemma}
  \theoremstyle{remark}
  \newtheorem{rem}[thm]{Remark}
  \theoremstyle{plain}
  \newtheorem{cor}[thm]{Corollary}

\makeatother

\newtheorem{theorem}{Theorem}[section]
\newtheorem{lemma}[theorem]{Lemma}

\newcommand{\be}{\begin{equation}}\newcommand{\ee}{\end{equation}}

\newcommand{\norm}[1]{\Vert #1 \Vert}

\makeatother

\usepackage{babel}

\begin{document}

\title{Absolute Continuity, Interpolation and the Lyapunov Order}

\author{Paul S. Muhly and Baruch Solel}

\address{Department of Mathematics\\
 University of Iowa\\
 Iowa City, IA 52242}

\email{paul-muhly@uiowa.edu}

\address{Department of Mathematics\\
 Technion\\
 32000 Haifa\\
 Israel}

\email{mabaruch@techunix.technion.ac.il}


\thanks{The research of both authors was supported in part by a U.S.-Israel
Binational Science Foundation grant. The second author was also supported by the Technion V.P.R. Fund.}

\begin{abstract}
We extend our Nevanlinna-Pick theorem for Hardy algebras and their
representations to cover interpolation at the absolutely continuous
points of the boundaries of their discs of representations. The Lyapunov
order plays a crucial role in our analysis.
\end{abstract}
\maketitle

\section{Introduction}

The celebrated theorem of Nevanlinna and Pick asserts that if $n$
distinct points, $z_{1},z_{2},\ldots,z_{n}$, are given in the open
unit disc $\mathbb{D}$ and if $n$ other complex numbers are also
given, $w_{1},w_{2},\ldots,w_{n}$, then there is a function $f$
in the Hardy algebra $H^{\infty}(\mathbb{T})$, with norm at most
one, such that $f(z_{i})=w_{i}$, $i=1,2,\ldots,n$, if and only if
the Pick matrix \[
\left(\frac{1-w_{i}\overline{w_{j}}}{1-z_{i}\overline{z_{j}}}\right)_{i,j=1}^{n}\]
is positive semidefinite. In \cite[Theorem 5.3]{MS2004}, we generalized
this Nevanlinna-Pick theorem to the setting of Hardy algebras over
$W^{*}$-correspondences. Here we intend to push the work in \cite{MS2004}
further, using tools developed in \cite{MS2010}. In a sense that
we shall make precise, we show that there is a condition similar to
the positivity of the Pick matrix that allows one to interpolate at
``absolutely continuous'' points of the boundaries of the domains
considered in \cite{MS2004}. Before stating our main theorem, we
want to see how one might try to extend the Nevanlinna-Pick theorem
to cover families of absolutely continuous contraction operators along
the lines suggested by \cite[Theorem 6.1]{MS2004}.

For this purpose, suppose $H$ is a Hilbert space and $Z:=(Z_{1},Z_{2},\ldots,Z_{n})$
is an $n$-tuple of operators in $B(H)$. Then $Z$ defines a completely
positive operator $\Phi_{Z}$ on the $n\times n$ matrices over $B(H)$
via the formula\begin{align*}
\Phi_{Z}((a_{ij})): & =\left[\begin{array}{cccc}
Z_{1}\\
 & Z_{2}\\
 &  & \ddots\\
 &  &  & Z_{n}\end{array}\right]\left[\begin{array}{cccc}
a_{11} & a_{12} & \cdots & a_{1n}\\
a_{21} & a_{22}\\
\vdots &  & \ddots\\
a_{n1} & a_{n2} & \cdots & a_{nn}\end{array}\right]\left[\begin{array}{cccc}
Z_{1}\\
 & Z_{2}\\
 &  & \ddots\\
 &  &  & Z_{n}\end{array}\right]^{*}\\
= & (Z_{i}a_{ij}Z_{j}^{*}).\end{align*}
If all the $Z_{i}$'s have norm less than $1$, then $I-\Phi_{Z}$
is an invertible map on $M_{n}(B(H))$ and $(I-\Phi_{Z})^{-1}$ is
also completely positive. The following theorem, then, is a special
case of \cite[Theorem 6.1]{MS2004}.
\begin{thm}
\label{thm:Classical-NP} Suppose $Z_{1},Z_{2},\cdots,Z_{n}$ are
$n$ distinct operators in $B(H)$, each of norm less than $1$, and
suppose $W_{1},W_{2},\cdots,W_{n}$ are $n$ operators in $B(H)$.
Then there is an function $f$ in $H^{\infty}(\mathbb{T})$, with
supremum norm at most $1$, such that $f(Z_{i})=W_{i}$, $i=1,2,\ldots,n$,
where $f(Z_{i})$ is defined through the Riesz functional calculus,
if and only if the Pick map\begin{equation}
(I-\Phi_{W})\circ(I-\Phi_{Z})^{-1}\label{eq:NP-positivity}\end{equation}
defined on $M_{n}(B(H))$ is completely positive.
\end{thm}
Observe that when $H$ is one-dimensional this theorem recovers the
classical theorem of Nevanlinna and Pick that we stated at the outset.
Now Sz.-Nagy and Foia\c{s} have shown that the proper domain for
their $H^{\infty}$-functional calculus is the collection of all \emph{absolutely
continuous contractions}. One way to say that a contraction $T$ is
absolutely continuous is to say that when $T$ is decomposed as $T=T_{cnu}+U$,
where $T_{cnu}$ is completely non-unitary and $U$ is unitary, then
the spectral measure of $U$ is absolutely continuous with respect
to Lebesgue measure on the circle. The content of \cite[Theorems III.2.1 and III.2.3]{Sz-NFBK2010}
is that a contraction $T$ is absolutely continuous if and only if
the $H^{\infty}(\mathbb{T})$-functional calculus may be evaluated
on $T$. It is therefore of interest to modify the hypothesis of Theorem
\ref{thm:Classical-NP} and ask for conditions that allow one to interpolate
in the wider context where the variables $Z_{i}$ are assumed to be
merely absolutely continuous contractions. In that setting the map
$I-\Phi_{Z}$ need no longer be invertible, and so it may not be possible
to form the generalized Pick operator $(I-\Phi_{W})\circ(I-\Phi_{Z})^{-1}$,
let alone determine whether or not it is completely positive. However,
there is a notion from matrix analysis, called the ``Lyapunov order'',
which suggests a replacement for the condition that the Pick operator
$(I-\Phi_{W})\circ(I-\Phi_{Z})^{-1}$be completely positive. To formulate
it we require an idea from the theory of completely positive maps
that we analyzed in \cite{MS2010}.
\begin{defn}
\label{def:_Superharmonic} Let $\Phi$ be a completely positive map
on a $W^{*}$-algebra $A$. An element $a\in A$ is called \emph{superharmonic}
for $\Phi$ in case $a\geq0$ and $\Phi(a)\leq a$; $a$ is called
\emph{pure superharmonic} in case $a$ is superharmonic and $\Phi^{n}(a)\searrow0$
as $n\to\infty$.
\end{defn}
The superharmonic elements for a completely positive map evidently
form a convex subset in the cone of all non-negative elements in the
$W^{*}$-algebra $A$.
\begin{defn}
\label{def:_Lyapunov_domination} Let $B$ be a $W^{*}$-algebra and
suppose $A$ is a sub-$W^{*}$-algebra of $B$. Suppose $\Phi:A\to A$
is a completely positive map and that $\Psi:B\to B$ is also completely
positive. Then we say $\Psi$ \emph{completely dominates} $\Phi$
\emph{in the sense of Lyapunov} in case every pure superharmonic element
of $M_{n}(A)$ for $\Phi_{n}$ is superharmonic for $\Psi_{n}$, where
$\Phi_{n}$ (resp. $\Psi_{n}$) is the usual promotion of $\Phi$
(resp. $\Psi$) to $M_{n}(A)$ (resp. $M_{n}(B)$).

The following proposition links the notion of complete Lyapunov domination
to the complete positivity of \eqref{eq:NP-positivity}.\end{defn}
\begin{prop}
\label{pro:NP_vs_Lyapunov} Suppose that $A$ is a sub-$W^{*}$-algebra
of a $W^{*}$-algebra $B$ and suppose $\Phi$ and $\Psi$ are completely
positive maps on $A$ and $B$, respectively. Assume that $\Vert\Phi\Vert<1$,
so $I-\Phi$ is invertible. Then the Pick operator, $P:=(I-\Psi)\circ(I-\Phi)^{-1}$,
is completely positive if and only if $\Psi$ completely dominates
$\Phi$ in the sense of Lyapunov.\end{prop}
\begin{proof}
Note that the hypothesis that $\Vert\Phi\Vert<1$ implies that every
superharmonic element of $A$ is pure superharmonic. Also note that
it suffices to prove that $P$ is positive if and only if $\{a\in A\mid a\geq0,\,\,\Phi(a)\leq a\}\subseteq\{b\in B\mid b\geq0,\,\,\Psi(b)\leq b\}$,
since the same argument will work for every $n$. Suppose, then, that
$P$ is positive and suppose that $a\geq0$ and $\Phi(a)\leq a$.
Then $(I-\Phi)(a)\geq0$. Consequently, $0\leq P((I-\Phi)(a))=(I-\Psi)(a)$,
showing that $a\geq\Psi(a)$. Suppose, conversely, that $b\geq0$.
Then since $\Vert\Phi\Vert<1$ and $\Phi$ is positive, $(I-\Phi)^{-1}=\sum_{n\geq0}\Phi^{n}$
is positive. Consequently, $a=(I-\Phi)^{-1}(b)$ is positive. But
also, since $(I-\Phi)(a)=b$ is positive, $a\geq\Psi(a)$, by hypothesis.
That is, $P(b)=a-\Psi(a)\geq0$, which is what we want to show.
\end{proof}
Our extension of Theorem \ref{thm:Classical-NP} can now be formulated
as
\begin{thm}
\label{thm:Classical_NP-revised} Suppose $Z_{1},Z_{2},\cdots,Z_{n}$
are $n$ distinct absolutely continuous contractions on a Hilbert
space $H$ and suppose $W_{1},W_{2},\cdots,W_{n}$ are $n$ contractions
on $H$, then there is a function $f\in H^{\infty}(\mathbb{T})$,
of norm at most $1$, such that $f(Z_{i})=W_{i}$, $i=1,2,\ldots,n$,
if and only if $\Phi_{W}$ completely dominates $\Phi_{Z}$ in the
sense of Lyapunov.
\end{thm}
The technology we use to prove Theorem \ref{thm:Classical_NP-revised}
works in the more general context of Hardy algebras over $W^{*}$-correspondences,
as we mentioned earlier. This is the arena in which our analysis takes
place. But first, we must provide some background from \cite{MS2004,MS2010}.
We shall follow terminology and most of the notation from \cite{MS2010}.
In particular, we shall cite the second section of \cite{MS2010}
for further background because it gives a fairly detailed birds-eye
view of the theory as of 2010.

\specialsection*{Acknowledgment}

We are very grateful to Nir Cohen for introducing us to the Lyapunov
order.

\section{Background and the Main Theorem}

Throughout this note, $M$ will denote a fixed $W^{*}$-algebra. We
will treat $M$ as an abstract $C^{*}$-algebra that is a dual space
and we will not think of it as acting concretely on Hilbert space
except through representations that we will specify. Also, $E$ will
denote a $W^{*}$-correspondence over $M$. This means first that
$E$ is a right Hilbert $C^{*}$-module over $M$ which is self-dual.
Consequently, the algebra of all bounded adjointable $M$-module maps
on $E$, $\mathcal{L}(E)$, is all the bounded module maps and $\mathcal{L}(E)$
is a $W^{*}$-algebra. To say that $E$ is a $W^{*}$-correspondence
over $M$ means, then, that there is a normal representation $\varphi:M\to\mathcal{L}(E)$,
making $E$ a left $M$-module \cite[Paragraph 2.2]{MS2010}. To eliminate
technical digressions we assume that $\varphi$ is faithful and unital.
The tensor powers of $E$, $E^{\otimes n}$, will be the self-dual
completions of the usual $C^{*}$-Hilbert module tensor powers, and
the Fock space $\mathcal{F}(E)$ will be the self-dual completion
of the $C^{*}$-direct sum of the $E^{\otimes n}$. Then $\mathcal{F}(E)$
is a $W^{*}$-correspondence over $M$ and we denote by $\varphi_{\infty}$
the left action of $M$ on $\mathcal{F}(E)$ \cite[Paragraph 2.7]{MS2010}.
If $\xi\in E$, then $T_{\xi}$ will denote the creation operator
it determines: $T_{\xi}\eta:=\xi\otimes\eta$, $\eta\in\mathcal{F}(E)$.
The norm-closed subalgebra generated $\varphi_{\infty}(M)$ and $\{T_{\xi}\mid\xi\in E\}$
is called the \emph{tensor algebra} of $E$ and will be denoted by
$\mathcal{T}_{+}(E)$ \cite[Paragraph 2.7]{MS2010}. The ultra weak
closure of $\mathcal{T}_{+}(E)$ in $\mathcal{L}(\mathcal{F}(E))$
is called the \emph{Hardy algebra of} $E$ and is denoted $H^{\infty}(E)$
\cite[Definition 2.1]{MS2010}.

Suppose $\sigma:M\to B(H_{\sigma})$ is a normal representation and
let $\sigma^{E}:\mathcal{L}(E)\to B(E\otimes_{\sigma}H_{\sigma})$
be the induced representation of $\mathcal{L}(E)$ in the sense of
Rieffel \cite{mR74a,mR74b}: $\sigma^{E}(T):=T\otimes I$, $T\in\mathcal{L}(E)$.
We write $\mathcal{I}(\sigma^{E}\circ\varphi,\sigma)$ for the set
of all operators $C:E\otimes_{\sigma}H_{\sigma}\to H_{\sigma}$ that
satisfy the equation $C\sigma\circ\varphi(a)=\sigma(a)C$ for all
$a\in M$; i.e., $\mathcal{I}(\sigma^{E}\circ\varphi,\sigma)$ denotes
all the \emph{intertwiners} of $\sigma^{E}\circ\varphi$ and $\sigma$.
Also, we write $\mathbb{D}(E,\sigma)$ for the set of all elements
of $\mathcal{I}(\sigma^{E}\circ\varphi,\sigma)$ that have norm less
than $1$, and we write $\overline{\mathbb{D}(E,\sigma)}$ for its
norm closure. In \cite{MS98b} we proved
\begin{lem}
\label{lem:Representations}(See \cite[Paragraph 2.8]{MS2010}.) Given
$\mathfrak{z}\in\overline{\mathbb{D}(E,\sigma)}$ , define $\mathfrak{z}\times\sigma$
by $\mathfrak{z}\times\sigma(\varphi_{\infty}(a)):=\sigma(a)$ and
$\mathfrak{z}\times\sigma(T_{\xi})(h):=\mathfrak{z}(\xi\otimes h)$,
$a\in M$, $\xi\in E$, and $h\in H_{\sigma}$. Then $\mathfrak{z}\times\sigma$
extends to a completely contractive (c.c.) representation of $\mathcal{T}_{+}(E)$
on $H_{\sigma}$. Conversely, given a c.c. representation $\rho$
of $\mathcal{T}_{+}(E)$, then $\rho=\mathfrak{z}\times\sigma$, where
$\sigma:=\rho\circ\varphi_{\infty}$ and $\mathfrak{z}(\xi\otimes h):=\rho(T_{\xi})h$.
Further, for $F\in H^{\infty}(E)$, the $B(H_{\sigma})$-valued function
$\widehat{F}_{\sigma}$, defined on $\mathbb{D}(E,\sigma)$ by $\widehat{F}_{\sigma}(\mathfrak{z}):=\mathfrak{z}\times\sigma(F)$,
is bounded analytic and it extends to be continuous on $\overline{\mathbb{D}(E,\sigma)}$
when $F\in\mathcal{T}_{+}(E)$.\end{lem}
\begin{rem}
\label{rem:Notation_change} We note here that our $\mathbb{D}(E,\sigma)$
is denoted $\mathbb{D}(E^{\sigma})^{*}$ in \cite[Paragraph 2.8]{MS2010},
where $E^{\sigma}:=\mathcal{I}(\sigma^{E}\circ\varphi,\sigma)^{*}=\mathcal{I}(\sigma,\sigma^{E}\circ\varphi)$
is the \emph{$\sigma$-dual of} $E$ \cite[Paragraph 2.6]{MS2010}.
This dual space plays an important role in our theory, as we shall
see, but we have opted for the change of notation in order to eliminate
numerous unnecessary and often confusing adjoints from our formulas. \end{rem}
\begin{defn}
\label{def:absolutely_continuous}A point $\mathfrak{z}\in\overline{\mathbb{D}(E,\sigma)}$
and the representation $\mathfrak{z}\times\sigma$ are called \emph{absolutely
continuous} in case $\mathfrak{z}\times\sigma$ extends to be an ultra
weakly continuous representation of $H^{\infty}(E)$ in $B(H_{\sigma})$.
We write $\mathcal{AC}(E,\sigma)$ for all the absolutely continuous
points of $\overline{\mathbb{D}(E,\sigma)}$.
\end{defn}
Our choice of terminology is inspired by the fact that when $M=E=\mathbb{C}$,
then $\mathfrak{z}$ is absolutely continuous in our sense if and
only if $\mathfrak{z}$, which is just an ordinary contraction operator
on $H_{\sigma}$, is absolutely continuous in the sense described
in the Introduction.

In general, $\mathbb{D}(E,\sigma)\subseteq\mathcal{AC}(E,\sigma)\subseteq\overline{\mathbb{D}(E,\sigma)}$,
and both inclusions are proper. If $M=E=\mathbb{C}$, and if $\sigma$
is the one-dimensional representation of $\mathbb{C}$ on $\mathbb{C}$,
then $\mathbb{D}(E,\sigma)$ is the open unit disc in the complex
plane and $\mathbb{D}(E,\sigma)=\mathcal{AC}(E,\sigma)$. In every
other setting of which we are aware, $\mathbb{D}(E,\sigma)\subsetneq\mathcal{AC}(E,\sigma)$.
Also, we know of no situation where $\overline{\mathbb{D}(E,\sigma)}=\mathcal{AC}(E,\sigma)$.
We have been able to identify $\mathcal{AC}(E,\sigma)$ explicitly
in numerous instances \cite[Sections 4 and 5]{MS2010} and we know
a lot about this space, but there is still much that remains mysterious.

The $\sigma$-dual of $E$, $E^{\sigma}:=\mathcal{I}(\sigma,\sigma^{E}\circ\varphi)$,
is important in this study for several reasons. The first is that
it is a $W^{*}$-correspondence over $\sigma(M)'$ in a very natural
way. For $\xi,\eta\in E^{\sigma}$, $\langle\xi,\eta\rangle$ is defined
to be $\xi^{*}\eta$ - the product being the ordinary operator product,
which makes sense as an operator on $H_{\sigma}$ since $\xi$ and
$\eta$ both map from $H_{\sigma}$ to $E\otimes_{\sigma}H_{\sigma}$.
The actions of $\sigma(M)'$ on $E^{\sigma}$ are given by the formula:\[
a\cdot\xi\cdot b:=(I_{E}\otimes a)\xi b,\quad a,b\in\sigma(M)',\,\xi\in E^{\sigma}.\]
Again, the products on the right hand side of the equation are ordinary
operator products. The concept of the $\sigma$- dual of a $W^{*}$-correspondence
was formalized in \cite{MS2004}, but it appeared, implicitly, in
a number of places. A key role that it will play here is in the identification
of the commutant of an induced representation, which we will describe
in the next section. But here we can already see its relevance for
the present considerations by virtue of the following observation:
Let $\mathfrak{z}_{1},\mathfrak{z}_{2},\ldots,\mathfrak{z}_{n}$ be
points in $\overline{\mathbb{D}(E,\sigma)}$. Then they define a map
$\Phi_{\mathfrak{z}}$ on the $n\times n$ matrices over $\sigma(M)'$
by the formula \begin{equation}
\Phi_{\mathfrak{z}}((a_{ij})):=\left(\langle\mathfrak{z}_{i},a_{ij}\cdot\mathfrak{z}_{j}\rangle\right)\quad(a_{ij})\in M_{n}(\sigma(M)').\label{eq:Phi_z}\end{equation}
A moment's reflection reveals that $\Phi_{\mathfrak{z}}$ is completely
positive, as it is the composition of manifestly completely positive
maps.

Our objective in this note is the proof of the following theorem,
which will occupy the next section.
\begin{thm}
\label{thm:Absolute_NP_Thm} Suppose $E$ is a $W^{*}$-correspondence
over a $W^{*}$-algebra $M$ and that $\sigma$ is a faithful normal
representation of $M$ on the Hilbert space $H_{\sigma}$. Suppose,
too, that $n$ distinct points $\mathfrak{z}_{1},\mathfrak{z_{2},}\ldots,\mathfrak{z}_{n}\in\mathcal{AC}(E,\sigma)$
are given and that $n$ operators in $B(H_{\sigma})$, $W_{1},W_{2},\cdots,W_{n}$,
are given. Define the map $\Phi_{\mathfrak{z}}$ on $M_{n}(\sigma(M)')$
by the formula $\Phi_{\mathfrak{z}}((a_{ij}))=(\langle\mathfrak{z}_{i},a_{ij}\cdot\mathfrak{z}_{j}\rangle)$
and define the map $\Phi_{W}$ on $M_{n}(B(H_{\sigma}))$ by the formula
$\Phi_{W}((T_{ij})):=\left(W_{i}T_{ij}W_{j}^{*}\right)$. Then there
is an element $F$ in $H^{\infty}(E)$, with $\Vert F\Vert\leq1$,
such that $\widehat{F}(\mathfrak{z}_{i})=W_{i}$, $i=1,2,\ldots,n$,
if and only if $\Phi_{W}$ completely dominates $\Phi_{\mathfrak{z}}$
in the sense of Lyapunov.
\end{thm}
\begin{proof}[Proof of Theorem \ref{thm:Classical_NP-revised}] That
theorem is an immediate consequence of Theorem \ref{thm:Absolute_NP_Thm}.
Indeed, in the setting of the former, $M=E=\mathbb{C}$, and $\sigma$
is just a multiple of the identity representation, the multiple being
the Hilbert space dimension of $H_{\sigma}$. Since we may safely
identify $\mathbb{C}\otimes_{\sigma}H_{\sigma}$ with $H_{\sigma}$,
$\overline{\mathbb{D}(E,\sigma)}$ may be identified with the closed
unit ball in $B(H_{\sigma})$, i.e., with all contractions on $H_{\sigma}$.
As we noted, the celebrated theorems of Sz.-Nagy and Foia\c{s} identify
$\mathcal{AC}(\mathbb{C},\sigma)$ with the absolutely continuous
contractions on $H_{\sigma}$ in the classical sense. When these identifications
are made, $\Phi_{\mathfrak{z}}$ of Theorem \ref{thm:Absolute_NP_Thm}
becomes the $\Phi_{Z}$ of Theorem \ref{thm:Classical_NP-revised}
and, of course, the two $\Phi_{W}$'s are the same. \end{proof}

\section{The Proof of Theorem \ref{thm:Absolute_NP_Thm}}

We will first show that if $\Phi_{W}$ completely dominates $\Phi_{\mathfrak{z}}$
in the sense of Lyapunov, then we can find an interpolating $F\in H^{\infty}(E)$
of norm at most $1$. The route we shall follow is similar, in certain
respects, to the route followed in the proof of \cite[Theorem 5.3]{MS2004}
and is based, ultimately, on the commutant lifting approach to the
classical Nevanlinna-Pick theorem pioneered by Sarason \cite{dS67}.
For this purpose, we need another way to express Lyapunov dominance
that reflects the fact that the $\mathfrak{z}_{i}$'s involved all
lie in $\mathcal{A}C(E,\sigma)$. The key tool in our approach is
the notion of an induced representation for $\mathcal{T}_{+}(E)$
and the connection such representations have with the concept of absolute
continuity. They are defined as follows: Let $\tau$ be a normal representation
of $M$ on the Hilbert space $H_{\tau}$. Then we may induce $\tau$
to $\mathcal{F}(E)$, obtaining a normal representation $\tau^{\mathcal{F}(E)}$
of $\mathcal{L}(\mathcal{F}(E))$ on the Hilbert space $\mathcal{F}(E)\otimes_{\tau}H_{\tau}$.
The restriction of $\tau^{\mathcal{F}(E)}$ to $\mathcal{T}_{+}(E)$,
then, is called the \emph{representation of} $\mathcal{T}_{+}(E)$
\emph{induced by} $\tau$. It is clearly an absolutely continuous
representation of $\mathcal{T}_{+}(E)$, since $H^{\infty}(E)$ is
contained in $\mathcal{L}(\mathcal{F}(E))$ by definition and $\tau^{\mathcal{F}(E)}$
is ultraweakly continuous. As we showed in \cite{MS2010}, and will
discuss in a moment, induced representations are the architypical
absolutely continuous representations. We continue to use the notation
$\tau^{\mathcal{F}(E)}$ for its restrictions to $\mathcal{T}_{+}(E)$
and $H^{\infty}(E)$.

In \cite{MS99} we develop at length the analogies between induced
representations of $\mathcal{T}_{+}(E)$ and $H^{\infty}(E)$ and
unilateral shifts. Indeed, a unilateral shift arises from an induced
representation of $\mathcal{T}_{+}(E)$, where $M=E=\mathbb{C}.$
\begin{defn}
Let $\pi$ be a faithful representation of $M$ on $H_{\pi}$ and
assume that $\pi$ has infinite multiplicity. Then $\pi^{\mathcal{F}(E)}$
is called \emph{the universal induced representation} of $\mathcal{T}_{+}(E)$
and $H^{\infty}(E)$ determined by $\pi$.

Any two faithful $\pi$'s with infinite multiplicity give unitarily
equivalent induced representations. Further, every induced representation
of $\mathcal{T}_{+}(E)$ is unitarily equivalent to a subrepresentation
of $\pi^{\mathcal{F}(E)}$ obtained by restricting $\pi^{\mathcal{F}(E)}$
to a subspace of the form $\mathcal{F}(E)\otimes_{\pi}\mathcal{K}$,
where $\mathcal{K}$ is a subspace of $H_{\pi}$ that reduces $\pi$
\cite[Paragraphs 2.5 and 2.11]{MS2010}. This explains the terminology,
allowing us to use the definite article. The representation $\pi$
and the induced representation $\pi^{\mathcal{F}(E)}$ will be fixed
for the remainder of this note.
\end{defn}
We also shall extend the notation $\mathcal{I}(\sigma^{E}\circ\varphi,\sigma)$,
and write $\mathcal{I}(\rho_{1},\rho_{2})$ for the set of operators
$C:H_{\rho_{1}}\to H_{\rho_{2}}$ that intertwine $\rho_{1}$ and
$\rho_{2}$, where $\rho_{1}$ and $\rho_{2}$ are any two completely
contractive representations of $\mathcal{T}_{+}(E)$. If $\rho_{i}$
is written as $\mathfrak{z}_{i}\times\sigma_{i}$, $i=1,2$, then
it is easy to see that an operator $C:H_{\rho_{1}}\to H_{\rho_{2}}$
lies in $\mathcal{I}(\rho_{1},\rho_{2})$ if and only if $C\in\mathcal{I}(\sigma_{1},\sigma_{2})$
and $\mathfrak{z}_{2}(I_{E}\otimes C)=C\mathfrak{z}_{1}$.

An important linkage among the universal induced representation, intertwiners,
and absolute continuity is the following theorem.
\begin{thm}
\cite[Theorem 4.7]{MS2010}\label{AbsCont} A point $\mathfrak{z}\in\overline{\mathbb{D}(E,\sigma)}$
is absolutely continuous if and only if \[
\bigvee\{\mbox{Ran}(c)\mid c\in\mathcal{I}(\pi^{\mathcal{F}(E)},\mathfrak{z}\times\sigma)\}=H_{\sigma}.\]
Further, for $F\in H^{\infty}(E)$, $\mathfrak{z}\in\mathcal{AC}(E,\sigma)$,
and $c\in\mathcal{I}(\pi^{\mathcal{F}(E)},\mathfrak{z}\times\sigma)$,
\begin{equation}
\widehat{F}(\mathfrak{z})c=c\pi^{\mathcal{F}(E)}(F).\label{eq:pt_evaluation}\end{equation}
\end{thm}
\begin{proof}
The first assertion is explicitly in \cite{MS2010} as Theorem 4.7.
The second assertion is easily checked on generators of $H^{\infty}(E)$
of the form $\varphi_{\infty}(a)$, $a\in M$, and $T_{\xi}$, $\xi\in E$.
That is all that is necessary to check.
\end{proof}
%
{}

Recall that if $\mathfrak{z}\in\overline{\mathbb{D}(E,\sigma)}$,
then $\mathfrak{z}^{*}$ lies in the $W^{*}$-correspondence $E^{\sigma}$
over $\sigma(M)'$. It therefore defines a completely positive map
$\Theta_{\mathfrak{z}}$ on $\sigma(M)'$ by the formula \[
\Theta_{\mathfrak{z}}(a)=\langle\mathfrak{z}^{*},a\cdot\mathfrak{z}^{*}\rangle=\mathfrak{z}(I_{E}\otimes a)\mathfrak{z}^{*},\quad a\in\sigma(M)'.\]
Indeed, $\Theta_{\mathfrak{z}}$ is just a special case of the map
$\Phi_{\mathfrak{z}}$ in the statement of Theorem \ref{thm:Absolute_NP_Thm}.
We are going to use the following theorem from \cite{MS2010} to obtain
an alternate formulation of the complete Lyapunov dominance assertion
in that theorem.
\begin{thm}
\cite[Theorem 4.6]{MS2010}\label{thm:factoring_superharm_ops} If
$\mathfrak{z}\in\overline{\mathbb{D}(E,\sigma)}$, then an operator
$q\in\sigma(M)'$ is a pure superharmonic operator for $\Theta_{\mathfrak{z}}$
if and only if $q$ can be written as $q=cc^{*}$ for an element $c\in\mathcal{I}(\pi^{\mathcal{F}(E)},\mathfrak{z}\times\sigma)$. \end{thm}
\begin{cor}
\label{cor:Alternate_Lyapunov_domination}We adopt the notation of
Theorem \ref{thm:Absolute_NP_Thm}. The map $\Phi_{W}$ completely
dominates $\Phi_{\mathfrak{z}}$ in the sense of Lyapunov if and only
if the following condition is satisfied: For every integer $m\geq1$,
for every choice of function $l:\{1,2,\ldots,m\}\to\{1,2,\ldots,n\},$
and for any choice of $m$ operators $c_{j}\in\mathcal{I}(\pi^{\mathcal{F}(E)},\mathfrak{z}_{l(j)}\times\sigma)$
the operator matrix inequality\begin{equation}
(W_{l(i)}c_{i}c_{j}^{*}W_{l(j)}^{*})_{i,j=1}^{m}\leq(c_{i}c_{j}^{*})_{i,j=1}^{m}\label{eq:NPCond1}\end{equation}
is satisfied.\end{cor}
\begin{proof}
Fix $m$ and a function $l:\{1,2,\ldots,m\}\to\{1,2,\ldots,n\}$.
Write $H_{\sigma}^{(m)}$ for the direct sum of $m$ copies of $H_{\sigma}$
and let $\sigma_{m}$ be the inflated representation of $M$ on $H_{\sigma}^{(m)}$,
i.e., $\sigma_{m}(a):=\mbox{diag}\{a,a,\cdots,a\}$. Then $\sigma_{m}(M)'=M_{m}(\sigma(M)')$.
Also, write $E^{(m)}$ for the direct sum of copies of $E$, which
is also a $W^{*}$-correspondence over $M$ in the obvious way, and
set $\tilde{\mathfrak{z}}:=(\mathfrak{z}_{l(1)},\mathfrak{z}_{l(2)},\cdots,\mathfrak{z}_{l(m)})$.
Then we may view $\tilde{\mathfrak{z}}$ as a map from $E^{(m)}\otimes_{\sigma_{m}}H_{\sigma}^{(m)}=(E\otimes_{\sigma}H_{\sigma})^{(m)}$
to $H_{\sigma}^{(m)}$, which clearly belongs to $\mathcal{I}(\sigma_{m}^{E^{(m)}}\circ\varphi,\sigma_{m})$.
Consequently, $\tilde{\mathfrak{z}}$ defines a completely positive
map $\Theta_{\tilde{\mathfrak{z}}}$ on $\sigma_{m}(M)'=M_{m}(\sigma(M)')$,
and it is easy to see that \[
\Theta_{\tilde{\mathfrak{z}}}((b_{l(i),l(j)}))_{i,j=1}^{m}=\left(\mathfrak{z}_{l(i)}(I_{E}\otimes b_{l(i),l(j)})\mathfrak{z}_{l(j)}^{*}\right),\quad(b_{l(i),l(j)})_{i,j=1}^{m}\in M_{m}(\sigma(M)').\]
Moreover, by Theorem \ref{thm:factoring_superharm_ops}, the pure
superharmonic elements $M_{m}(\sigma(M)')$ for $\Theta_{\tilde{\mathfrak{z}}}$
are of the form $(c_{i}c_{j}^{*})_{i,j=1}^{m}$, where $c_{i}\in\mathcal{I}(\pi^{\mathcal{F}(E)},(\mathfrak{z}_{l(i)}\times\sigma))$.
On the other hand, the $W_{i}$'s may be used to define the completely
positive map $\Psi_{W,l}$ on $M_{m}(\sigma(M)')$ by the formula\[
\Psi_{W,l}((b_{l(i),l(j)})):=\left(W_{l(i)}b_{l(i),l(j)}W_{l(j)}^{*}\right),\quad(b_{l(i),l(j)})_{i,j=1}^{m}\in M_{m}(\sigma(M)').\]
The inequality \ref{eq:NPCond1} is the statement that the condition
of the corollary is equivalent to the assertion that $\Psi_{W,l}$
dominates $\Theta_{\tilde{\mathfrak{z}}}$ in the sense of Lyapunov
for each choice of $m$ and $l$. It is now evident that the condition
of the corollary implies that $\Phi_{W}$ \emph{completely} dominates
$\Phi_{\mathfrak{z}}$ in the sense of Lyapunov by choosing $m$ and
$l$ judiciously. On the other hand, if $\Phi_{W}$ completely dominates
$\Phi_{\mathfrak{z}}$, then given $m$ and $l$, one can clearly
choose a $k$ so that the domination of $(\Phi_{\mathfrak{z}})_{k}$by
$(\Phi_{W})_{k}$ in the sense of Lyapunov gives the desired inequalities
of the condition for that $m$ and $l$.
\end{proof}
In order to follow the commutant lifting approach pioneered by Sarason,
we require the description of the commutant of $\pi^{\mathcal{F}(E)}(H^{\infty}(E))$
that we developed in \cite{MS2004}. The description there works for
any induced representation, but we formulate it here specifically for
$\pi^{\mathcal{F}(E)}$.
\begin{thm}
\cite[Theorem 3.9]{MS2004}\label{thm:Commutant_Induced} Write $\iota$
for the identity representation of $\pi(M)'$ on $H_{\pi}$, and let
$\tau$ be the induced representation of $\mathcal{L}(E^{\pi})$ acting
on $\mathcal{F}(E^{\pi})\otimes_{\iota}H_{\pi}$, i.e., let $\tau=\iota^{\mathcal{F}(E^{\pi})}$.
Then the map $U:\mathcal{F}(E^{\pi})\otimes_{\iota}H_{\pi}\to\mathcal{F}(E)\otimes_{\pi}H_{\pi}$,
defined by the formula \begin{equation}
U(\xi_{1}\otimes\xi_{2}\otimes\cdots\otimes\xi_{n}\otimes h):=(I_{E^{\otimes(n-1)}}\otimes\xi_{1})(I_{E^{\otimes(n-2)}}\otimes\xi_{2})\cdots(I_{E}\otimes\xi_{n-1})\xi_{n}(h),\label{eq:Def_of_U}\end{equation}
$\xi_{1}\otimes\xi_{2}\otimes\cdots\otimes\xi_{n}\otimes h\in(E^{\pi})^{\otimes n}\otimes_{\iota}H_{\pi}$,
is a Hilbert space isomorphism and \[
U\tau(H^{\infty}(E^{\pi}))U^{*}=\pi^{\mathcal{F}(E)}(H^{\infty}(E))'.\]
Likewise,
$U^{*}\pi^{\mathcal{F}(E)}(H^{\infty}(E))U=\tau(H^{\infty}(E^{\pi}))'$,
and the double commutant relations hold:\[
\pi^{\mathcal{F}(E)}(H^{\infty}(E))''=\pi^{\mathcal{F}(E)}(H^{\infty}(E)),\qquad\mbox{and}\qquad\tau(H^{\infty}(E^{\pi}))''=\tau(H^{\infty}(E^{\pi})).\]

\end{thm}
We are now ready to show how the complete domination of $\Phi_{\mathfrak{z}}$
in the sense of Lyapunov by $\Phi_{W}$ implies that we can interpolate
the $W$'s at the $\mathfrak{z}$'s in Theorem \ref{thm:Absolute_NP_Thm}.
\begin{lem}
Let \[
\mathcal{M}=\overline{span}\{U^{*}c^{*}h\mid h\in H_{\sigma},\; c\in\mathcal{I}(\pi^{\mathcal{F}(E)},\mathfrak{z}_{i}\times\sigma),\;1\leq i\leq n\}.\]
 Then $\mathcal{M}$ is a closed subspace of $\mathcal{F}(E^{\pi})\otimes_{\iota}H_{\pi}$
that is invariant under $\tau(H^{\infty}(E^{\pi}))^{*}$.\end{lem}
\begin{proof}
For $X\in\tau(H^{\infty}(E^{\pi}))$, $U\tau(X)U^{*}$ lies in the
commutant of $\pi^{\mathcal{F}(E)}(H^{\infty}(E))$ by Theorem \eqref{thm:Commutant_Induced}.
Consequently $cU\tau(X)U^{*}\in\mathcal{I}(\pi^{\mathcal{F}(E)},\mathfrak{z}_{i}\times\sigma)$
for every $c\in\mathcal{I}(\pi^{\mathcal{F}(E)},\mathfrak{z}_{i}\times\sigma)$.
But then $\tau(X)^{*}U^{*}c^{*}h=U^{*}(U\tau(X)^{*}U^{*})c^{*}h=U^{*}(cU\tau(X)U^{*})^{*}h$
lies in $\mathcal{M}$ for all $U^{*}c^{*}h\in\mathcal{M}$.\end{proof}
\begin{lem}
\label{R} The correspondence, $U^{*}c^{*}h\to U^{*}c^{*}W_{i}^{*}$,
$c\in\mathcal{I}(\pi^{\mathcal{F}(E)},\mathfrak{z}_{i}\times\sigma)$,
defined on the generators of $\mathcal{M}$ extends to a well-defined
contraction operator on $\mathcal{M}$, say $R$, if and only if for
every integer $m\geq1$, for every choice of function $l:\{1,2,\ldots,m\}\to\{1,2,\ldots,n\},$
and for every choice of $m$ operators $c_{j}\in\mathcal{I}(\pi^{\mathcal{F}(E)},\mathfrak{z}_{l(j)}\times\sigma)$
the operator matrix inequality\[
(W_{l(i)}c_{i}c_{j}^{*}W_{l(j)}^{*})_{i,j=1}^{m}\leq(c_{i}c_{j}^{*})_{i,j=1}^{m}\]
is satisfied. In this event, $R$ commutes with the restriction of
$\tau(H^{\infty}(E^{\pi}))^{*}$to $\mathcal{M}$.\end{lem}
\begin{proof}
A linear combination of generators of $\mathcal{M}$ is a vector of
the form $k=\sum_{j=1}^{m}U^{*}c_{j}^{*}h_{j}$, where $c_{j}\in\mathcal{I}(\pi^{\mathcal{F}(E)},\mathfrak{z}_{l(j)}\times\sigma)$
for some $m$ and function $l:\{1,2,\ldots,m\}\to\{1,2,\ldots n\}$.
Since \[
\norm{k}^{2}=\sum_{j,i}\langle c_{i}c_{j}^{*}h_{j},h_{i}\rangle,\]
 while \[
\norm{\sum_{j=1}^{m}U^{*}c_{j}^{*}W_{l(j)}^{*}h_{j}}^{2}=\sum_{i,j}\langle W_{i(l)}c_{i}c_{j}^{*}W_{i(j)}^{*}h_{j},h_{l}\rangle,\]
the first assertion is immediate. But the second is also immediate
since $R$ is ``right multiplication'' by $W_{i}^{*}$ on a generator
of the form $U^{*}c^{*}h$, $c\in\mathcal{I}(\pi^{\mathcal{F}(E)},\mathfrak{z}_{i}\times\sigma)$,
i.e., $RU^{*}c^{*}h=U^{*}c^{*}W_{i}^{*}h$, while the restriction
of $\tau(X)^{*}$ to $\mathcal{M}$ acts by left multiplication for
all $X\in H^{\infty}(E^{\pi})$: $\tau(X)^{*}U^{*}c^{*}h=U^{*}(U\tau(X)^{*}U^{*})c^{*}h$.
\end{proof}
Since $\mathcal{M}$ is invariant for $\tau(H^{\infty}(E^{\pi}))^{*}$,
we obtain an ultra weakly continuous completely contractive representation
$\rho$ of $H^{\infty}(E^{\pi})$ on $\mathcal{M}$ by compressing
$\tau(H^{\infty}(E^{\pi}))$ to $\mathcal{M}$, i.e., \[
\rho(X):=P_{\mathcal{M}}\tau(X)\vert\mathcal{M},\qquad X\in H^{\infty}(E^{\pi}).\]
Since $\tau$ is isometric in the sense of \cite{MS98b} and since
$R^{*}$ commutes with $\rho(H^{\infty}(E^{\pi}))$, we may apply
our commutant lifting theorem \cite[Theorem 4.4]{MS98b} to conclude
that there is an operator $Y\in B(\mathcal{F}(E^{\pi})\otimes_{\iota}H_{\pi})$
of norm at most one such that $P_{\mathcal{M}}Y\vert\mathcal{M}=R^{*}$,
$Y\mathcal{M}^{\perp}\subseteq\mathcal{M}^{\perp}$, and $Y$ commutes
with $\tau(H^{\infty}(E^{\pi}))$ (see \cite[Theorem 2.6]{MS2010},
also). By Theorem \ref{thm:Commutant_Induced}, there is an $F\in H^{\infty}(E)$,
$\Vert F\Vert\leq1,$ such that $Y=U^{*}\pi^{\mathcal{F}(E)}(F)U$.
We conclude from the properties of $Y$ and the definition of $R$
that \[
U^{*}\pi^{\mathcal{F}(E)}(F)^{*}c^{*}h=(U^{*}\pi^{\mathcal{F}(E)}(F)^{*}U)U^{*}c^{*}h=Y^{*}U^{*}c^{*}h=RU^{*}c^{*}h=U^{*}c^{*}W_{i}^{*}h\]
for all $c\in\mathcal{I}(\pi^{\mathcal{F}(E)},\mathfrak{z}_{i}\times\sigma)$.
This, in turn, implies that\[
c\pi^{\mathcal{F}(E)}(F)=W_{i}c\]
for all such $c$. But $c\pi^{\mathcal{F}(E)}(F)=\widehat{F}(\mathfrak{z}_{i})c$
for all $c\in\mathcal{I}(\pi^{\mathcal{F}(E)},\mathfrak{z}_{i}\times\sigma)$,
by equation \eqref{eq:pt_evaluation} in Theorem \ref{AbsCont}. Therefore,\[
\widehat{F}(\mathfrak{z}_{i})c=W_{i}c\]
 for all $i$ and all $c\in\mathcal{I}(\pi^{\mathcal{F}(E)},\mathfrak{z}_{i}\times\sigma)$.
However, by hypothesis, all the $\mathfrak{z}_{i}$ lie in $\mathcal{AC}(E,\sigma)$.
Consequently, by the first assertion of Theorem \ref{AbsCont}, the
closed span of the ranges of the $c$'s in $\mathcal{I}(\pi^{\mathcal{F}(E)},\mathfrak{z}_{i}\times\sigma)$
is all of $H_{\sigma}$, for every $i$. We conclude that $\widehat{F}(\mathfrak{z}_{i})=W_{i}$
for every $i$. This completes the proof that if $\Phi_{W}$ completely
dominates $\Phi_{\mathfrak{z}}$ in the sense of Lyapunov, then there
is an $F\in H^{\infty}(E)$ that interpolates $W_{i}$ at $\mathfrak{z}_{i}$.
\begin{proof}
[Proof of the Converse] Part of the argument just given is reversible.
Suppose $F$ is an element of $H^{\infty}(E)$ of norm at most one
such that $\widehat{F}(\mathfrak{z}_{i})=W_{i}$, $i=1,2,\cdots,n$.
Then for each $c\in\mathcal{I}(\pi^{\mathcal{F}(E)},\mathfrak{z}_{i}\times\sigma)$
\[
c\pi^{\mathcal{F}(E)}(F)=\widehat{F}(\mathfrak{z}_{i})c=W_{i}c,\]
by equation \eqref{eq:pt_evaluation}. But then \[
(U^{*}\pi^{\mathcal{F}(E)}(F)^{*}U)U^{*}c^{*}=U^{*}c^{*}W_{i}^{*}\]
for all $c\in\mathcal{I}(\pi^{\mathcal{F}(E)},\mathfrak{z}_{i}\times\sigma)$.
Since the norm of $F$ is at most $1$ we conclude from Lemma \ref{R}
that for every integer $m\geq1$, for every choice of function $l:\{1,2,\ldots,m\}\to\{1,2,\ldots,n\},$
and for every choice of $m$ operators $c_{j}\in\mathcal{I}(\pi^{\mathcal{F}(E)},\mathfrak{z}_{l(j)}\times\sigma)$
the operator matrix inequality\[
(W_{l(i)}c_{i}c_{j}^{*}W_{l(j)}^{*})_{i,j=1}^{m}\leq(c_{i}c_{j}^{*})_{i,j=1}^{m}\]
is satisfied. So by Corollary \ref{cor:Alternate_Lyapunov_domination},
we conclude that $\Phi_{W}$ completely dominates $\Phi_{\mathfrak{z}}$
in the sense of Lyapunov.\end{proof}

\end{document}